%% file: Unrolled-Final.tex
\documentclass[12pt]{article}

\include{preamble}

\newcommand{\C}{\mathcal{C}}

\newcommand{\cartan}{\mathfrak{h}}

\newcommand{\Verma}[1]{M^{#1}}

\newcommand\doi[2]{\href{http://dx.doi.org/#1}{#2}}

\title{Categories of Weight Modules for Unrolled Restricted Quantum Groups at Roots of Unity}
\author{Matthew Rupert}
\date{}

\begin{document}
\maketitle

\begin{abstract}
Motivated by connections to the singlet vertex operator algebra in the $\mfg=\mathfrak{sl}_2$ case, we study the unrolled restricted quantum group $\overline{U}_q^H(\mfg)$ for any finite dimensional complex simple Lie algebra $\mfg$ at arbitrary roots of unity with a focus on its category of weight modules. We show that the braid group action naturally extends to the unrolled quantum groups and that the category of weight modules is a generically semi-simple ribbon category (previously known only for odd roots) with trivial M\"{u}ger center and self-dual projective modules.
\end{abstract}

\tableofcontents
 
\setlength{\parskip}{\baselineskip}%

\section{Introduction}

The unrolled quantum group $U^H_q(\mathfrak{sl}_2)$ was introduced in \cite[Subsection 6.3]{GPT1} inspired by previous work of Ohtsuki \cite{O}. This algebra was initially studied as an example for producing link invariants \cite{GPT1,GPT2}. In higher rank, the quantum groups $U^H_q(\mfg)$ associated to a simple finite dimensional complex Lie algebra $\mathfrak{g}$ have been studied at odd roots of unity, also mostly out of interest for their topological applications. It was shown in \cite{GP1,GP2} that the category $\C$ (see Definition \ref{weightmod}) of finite dimensional weight modules for $U^H_q(\mathfrak{g})$ at odd roots is ribbon and generically semi-simple (see Definition \ref{genss}) and connections between the unrolled quantum groups, knot invariants, and topological quantum field theories have been explored in \cite{BCGP,D,DGP}. Our primary motivation, however, is the expected connections between the unrolled restricted quantum group $\overline{U}_q^H(\mfg)$ at even roots of unity and a family of vertex operator algebras (VOAs) known as the higher rank singlet algebras. Although these quantum groups have interesting topological applications at any root of unity, connections to VOAs are only known to occur at even roots.

The (rank one) singlet vertex operator algebras $\mathcal{M}(p)$ are a very prominent family of VOAs studied by many authors \cite{A,AM1,AM3,AM4,CM2,CMW}. The fusion rules for simple $\mathcal{M}(p)$ modules are known for $p=2$, and there is a conjecture for $p>2$ \cite{CM1}. It was shown in \cite{CMR} that if the fusion rules for simple $\mathcal{M}(p)$-modules are as conjectured, then there exists an identification of simple $\overline{U}_q^H(\mathfrak{sl}_2)$ and $\mathcal{M}(p)$ modules which induces a ring isomorphism of the associated Grothendieck rings. Further, modified traces of the open Hopf links for $\overline{U}_q^H(\mathfrak{sl}_2)$ coincide exactly with the regularized asymptotic dimensions of characters for the singlet (see \cite[Theorem 1]{CMR}). There are a number of vertex operator algebras which can be constructed from the singlet, among them are the triplet $\mathcal{W}(p)$ and $B_p$ VOAs. It is possible to use the correspondence from \cite{CMR} between $\mathcal{M}(p)$ and $\overline{U}_q^H(\mathfrak{sl}_2)$ modules to construct new braided tensor categories which compare nicely to the module categories of the triplet and $B_p$ algebras \cite{CGR,ACKR}. The singlet, triplet, and $B_p$ algebras all have higher rank analogues denoted $\mathcal{W}^0(p)_Q$, $\mathcal{W}(p)_Q$, and $B(p)_Q$ respectively (see \cite{BM,FT,Mi,CM2} for $\mathcal{W}^0(p)_Q$ and $\mathcal{W}(p)_Q$, and \cite{C} for $B(p)_Q$), where $p \geq 2$ and $Q$ is the root lattice of a simple finite dimensional complex Lie algebra $\mfg$ of ADE type. As in the $\mathfrak{sl}_2$-case, their representation categories are expected to coincide with categories constructed from the category of weight modules of the corresponding unrolled restricted quantum group $\overline{U}_q^H(\mfg)$ at $2p$-th root of unity and some preliminary results in this direction can be found in \cite{FL,GLO,Len3,N}. Understanding this category is therefore prerequisite to many interesting problems relating to the $\mathcal{W}^0(p)_Q$, $\mathcal{W}(p)_Q$, and $B(p)_Q$ vertex algebras. The focus of this article is to study the category of weight modules $\C$ over $\overline{U}_q^H(\mfg)$ at arbitrary roots of unity, and extend results known for $\overline{U}_q^H(\mathfrak{sl}_2)$ at even roots and $\overline{U}_q^H(\mathfrak{g})$ at odd roots to this setting.

\subsection{Results}

Let $\mfg$ be a finite dimensional complex simple Lie algebra. Let $q$ be a primitive $\ell$-th root of unity, $r=\frac{3+(-1)^{\ell}}{4}\ell$ and $\mcC$ the category of finite dimensional weight modules for the unrolled restricted quantum group $\overline{U}_q^H(\mfg)$ (see Definition \ref{restricteddef} and the opening comments of Section \ref{Repsection}). We describe in Section \ref{sectionquantumgroup} how to construct the unrolled quantum groups as a semi-direct product $U_q^H(\mfg):= U_{q}(\mfg) \rtimes U(\mfh)$ of the De Concini - Kac specialization $U_q(\mfg)$ and the universal enveloping algebra of the Cartan subalgebra $\mfh$ of $\mfg$. We also show that the action of the braid group $\mcB_{\mfg}$ (see Definition \ref{braid}) extends naturally from $U_q(\mfg)$ to $U_q^H(\mfg)$ (Proposition \ref{dbraid}):
\begin{proposition}
The action of the braid group $\mathcal{B}_{\mfg}$ on $U_q(\mfg)$ can be extended naturally to the unrolled quantum group $U_q^H(\mfg)$.
\end{proposition}
This statement is known to some, but hasn't appeared in the literature to the author's knowledge. It has been shown previously (see \cite{CGP2}) that there is a generically semi-simple structure (see Definition \ref{genss}) on the category of weight modules over the restricted unrolled quantum group at odd roots of unity where $\ell \not \in 3\mathbb{Z}$ if $\mfg=G_2$. The purpose of this restriction on the $G_2$ case is to guarantee that $gcd(d_i,r)= 1$ where the $d_i=\frac{1}{2} \langle \alpha_i,\alpha_i \rangle$ are the integers symmetrizing the Cartan matrix. This condition fails at odd roots only for $G_2$, but for even roots all non ADE-type Lie algebras fail this condition for some choice of $\ell$. We show that when $gcd(d_i,r) \not = 1$, generic semi-simplicity can be retained if one quotients by a larger Hopf ideal (see Definition \ref{restricteddef}). In Subsection \ref{sectioncategoricaldata}, we observe that $\mcC$ being ribbon follows easily from the techniques developed in \cite{GP2} as in the case for odd roots, and we show that $\mcC$ has trivial M\"{u}ger center (see Definition \ref{muger}). We therefore have the following (Propositions \ref{gens}, \ref{Muger}, and Corollary \ref{ribbon}):
\begin{theorem}
$\mcC$ is a generically semi-simple ribbon category with trivial M\"{u}ger center.
\end{theorem}

$\overline{U}_q^H(\mathfrak{sl}_2)$ was studied at even roots of unity in \cite{CGP}. Stated therein  (\cite[Proposition 6.1]{CGP}) is a generator and relations description of the projective covers of irreducible modules. One apparent property of these projective covers is that their top and socle coincide. Showing that this is a general feature of projective covers in $\mcC$ is the topic of Subsection \ref{duality}. Let $L^{\lambda}$ denote the unique irreducible highest weight module of highest weight $\lambda \in \mfh^*$ and $P^{\lambda}$ its projective cover. We introduce a character preserving contravariant functor $M \mapsto \check{M}$, analogous to the duality functor for Lie algebras (\cite[Subsection 3.2]{H}). We are then able to prove the following theorem:

\begin{theorem}\label{selfdual}
The projective covers $P^{\lambda}$ are self-dual under the duality functor. That is, $\check{P^{\lambda}} \cong P^{\lambda}$.
\end{theorem}
This theorem has the following corollaries (Corollary \ref{projcorollary} and \ref{tracecorollary}):\\
\begin{corollary}\label{cor}
\begin{itemize}
\item $\mathrm{Socle}(P^{\lambda})=L^{\lambda}$.
\item $P^{\lambda}$ is the injective Hull of $L^{\lambda}$.
\item $\mcC$ is unimodular.
\item $\mcC$ admits a unique (up to scalar) two-sided trace on its ideal of projective modules.\\
\end{itemize}
\end{corollary}

\begin{acknowledgements} \quad \\
The author thanks Thomas Creutzig for many helpful comments and discussions.
\end{acknowledgements}

\section{\textbf{Preliminaries}}\label{prelim}

Let $\mathds{k}$ be a field. A $\mathds{k}$-category is a category $\mcC$ such that its Hom-sets are left $\mathds{k}$-modules, and morphism composition is $\mathds{k}$-bilinear. A pivotal $\mathds{k}$-linear category $\mcC$ (see \cite[Subsection 4.7]{EGNO} for a definition) is said to be $\mcG$-graded for some group $\mcG$ if for each $g \in \mcG$ we have a non-empty full subcategory $\mcC_g$ stable under retract such that 
\begin{itemize}
\item $\mcC= \bigoplus\limits_{g \in \mcG} \mcC_g$,
\item $V \in \mcC_g \implies V^* \in \mcC_{g^{-1}}$,
\item $V \in \mcC_g$ and $W \in \mcC_{g'} \implies V \otimes W \in \mcC_{gg'}$,
\item $V \in \mcC_g,W \in \mcC_{g'}$ and $\mathrm{Hom}(V,W) \not =0 \implies g=g'$.
\end{itemize}

where $V^*$ denotes a dual object of $V$. A subset $\mcX \subset \mcG$ is called symmetric if $\mcX^{-1}=\mcX$ and small if it cannot cover $\mcG$ by finitely many translations, i.e. for any $n \in \mathbb{N}$ and $\forall g_1,...,g_n \in \mcG$, $\bigcup\limits_{i=1}^n g_i \mcX \not = \mcG$.
\begin{definition}\label{genss} A $\mcG$-graded category $\mcC$ is called generically $\mcG$-semisimple if there exists a small symmetric subset $\mcX \subset \mcG$ such that for all $g \in \mcG \setminus \mcX$, $\mcC_g$ is semisimple. $\mcX$ is referred to as the singular locus of $\mcC$ and simple objects in $\mcC_g$ with $g \in \mcG \setminus \mcX$ are called generic. 
\end{definition}
Generically semisimple categories appeared in \cite{GP1,CGP} and were used in \cite{GP2} to prove that representation categories of unrolled quantum groups at odd roots of unity are ribbon. If $\mcC$ has braiding $c_{-,-}$, then an object $Y \in \mcC$ is said to be transparent if $c_{Y,X} \circ c_{X,Y} = \mathrm{Id}_{X \otimes Y}$ for all $X \in \mcC$. 
\begin{definition}\label{muger}
The M\"{u}ger center of $\C$ is the full subcategory of $\C$ consisting of all transparent objects.
\end{definition}
Triviality of the M\"{u}ger center should be viewed as a non-degeneracy condition. Indeed, for finite braided tensor categories triviality of the M\"{u}ger center is equivalent to the usual notions of non-degeneracy (see \cite[Theorem 1.1]{S}).

\section{The Unrolled Restricted Quantum Group $\overline{U}_q^H(\mfg)$}\label{sectionquantumgroup}

We first fix our notations. Let $\mfg$ be a simple finite dimensional complex Lie algebra of rank $n$ and dimension $n+2N$ where $N$ is the number of positive roots of $\mfg$ with Cartan martix $A=(a_{ij})_{1 \leq i,j \leq n}$ and Cartan subalgebra $\mfh$. Let $\Delta:=\{\alpha_1,...,\alpha_n\} \subset \mfh^*$ be a set of simple roots of $\mfg$, $\Delta^+$ ($\Delta^-$) the set of positive (negative) roots, and $L_R:=\bigoplus\limits_{i=1}^n \mbbZ \alpha_i$ the integer root lattice. Let $\{H_1,...,H_n\}$ be the basis of $\mfh$ such that $\alpha_j(H_i)=a_{ij}$ and $\langle , \rangle$ the form defined by $\langle \alpha_i,\alpha_j \rangle=d_ia_{ij}$ where $d_i=\langle \alpha_i, \alpha_i \rangle /2$ and normalized such that short roots have length 2. Let $L_W:= \bigoplus\limits_{i=1}^n \mbbZ \omega_i$ be the weight lattice generated by the dual basis $\{\omega_1,...,\omega_n\} \subset \mfh^*$ of $\{d_1H_1,...,d_nH_n\} \subset \mfh$, and $\rho:= \frac{1}{2} \sum\limits_{\alpha \in \Delta^+} \alpha \in L_W$ the Weyl vector.

Now, let $q \in \mathbb{C}^{\times}$ with $q_i=q^{d_i}$, and fix the notation
\begin{equation}\{x\}=q^x-q^{-x}, \quad  [x]=\frac{q^x-q^{-x}}{q-q^{-1}}, \quad [n]!=[n][n-1]...[1], \quad \binom{n}{m}=\frac{\{n\}!}{\{m\}!\{n-m\}!} ,
\end{equation}
\begin{equation}
 [j;q]=\frac{1-q^j}{1-q}, \qquad [j;q]!=[j;q][j-1;q] \cdots [1;q] \label{jq}.
\end{equation}
We will often use a subscript $i$, e.g. $[x]_i$, to denote the substitution $q \mapsto q_i$ in the above formulas.
 \begin{definition}\label{uqh}
Let $L$ be a lattice such that $L_R \subset L \subset L_W$. The unrolled quantum group $U^H_{q}(\mfg)$ associated to $L$ and $q \in \mathbb{C}^{\times}$ with $\mathrm{ord}(q^2)>d_i$ for all $i \in \{1,...,n\}$ is the $\mbbC$-algebra with generators $K_{\gamma},X_i,X_{-i},H_i$, (we will often let $K_{\alpha_i}:=K_i$) with $i=1,...,n$, $\gamma \in L$, and relations
\begin{align}
\label{eqKX}K_0=1, \quad K_{\gamma_1}K_{\gamma_2}=K_{\gamma_1+\gamma_2},\quad K_{\gamma}X_{\pm j}K_{-\gamma}=q^{\pm \langle \gamma,\alpha_j \rangle}X_{\pm j}, \\
\label{eqH} [H_i,H_j]=0, \quad [H_i,K_{\gamma}]=0, \qquad \quad [H_i,X_{\pm j}]=\pm a_{ij} X_{\pm j }, \quad \\
\label{eqX} \quad  [X_i,X_{-j}]=\delta_{i,j}\frac{K_{\alpha_j}-K_{\alpha_j}^{-1}}{q_j-q_j^{-1}},\qquad \qquad \qquad \\
\label{serre} \sum\limits_{k=0}^{1-a_{ij}} (-1)^k \binom{1-a_{ij}}{k}_{q_i} X^k_{\pm i}X_{\pm j} X_{\pm i}^{1-a_{ij}-k}=0 \qquad \text{if $ i \not = j$},
\end{align}
There is a Hopf-algebra structure on $U^H_{q}(\mfg)$ with coproduct $\Delta$, counit $\epsilon$, and antipode $S$ defined by
\begin{align}
\label{coK}\Delta(K_{\gamma})&=K_{\gamma} \otimes K_{\gamma}, & \epsilon(K_{\gamma})&=1,&  S(K_{\gamma})&=K_{-\gamma}\\
\label{coXi}\Delta(X_i)&=1 \otimes X_i + X_i \otimes K_{\alpha_i} & \epsilon(X_i)&=0, & S(X_i)&=-X_iK_{-\alpha_i},\\
\label{coX-i}\Delta(X_{-i})&=K_{-\alpha_i} \otimes X_{-i}+X_{-i} \otimes 1, & \epsilon(X_{-i})&=0, & S(X_{-i})&=-K_{\alpha_i}X_{-i}.,\\
\label{coH} \Delta(H_i)&=1\otimes H_i + H_i \otimes 1,& \quad \epsilon(H_i)&=0, &\quad S(H_i)&=-H_i.
\end{align}
\end{definition}
\begin{remark}
The cases $\mathrm{ord}(q^2) \not > d_i$ are investigated for the small quantum group in \cite{Len1}.
\end{remark}

The unrolled quantum group has been previously studied at roots of unity in relation to quantum topology and vertex algebras \cite{GP1,GP2,GPT1,GPT2,CGP,ACKR,CMR}. Interesting connections to Lusztig's quantum group of divided powers were established in \cite{Len2}. It is easy to see that when $q$ is a root of unity, the subalgebra generated by $K_{\gamma}$ and $X_{\pm i}$ is the De Concini - Kac specialization $U_{q}(\mfg)$. The unrolled quantum group is actually a smash product of the De Concini - Kac specialization with the universal enveloping algebra of $\mfh$, which we will briefly recall. Let the generators $H_1,...,H_n \in \mfh$ act on $U_{q}(\mfg)$ by the derivation $\partial_{H_i}:U_{q}(\mfg) \to U_{q}(\mfg)$ defined by
\begin{equation}\label{deriv} \partial_{H_i}X_{\pm j} = \pm a_{ij} X_{ \pm j}, \qquad \partial_{H_i}K_{\gamma} =0. \end{equation}
It is easy to see that these operators commute, so they do indeed define an action of $\mfh$ on $U_{q}(\mfg)$. It is easy to check the relations
\begin{align*}
(\Delta \circ \partial_{H_i})(X_j)&=(\mathrm{Id} \otimes \partial_{H_i} + \partial_{H_i} \otimes \mathrm{Id}) \circ \Delta (X_j) \\
\quad & \\
(\Delta \circ \partial_{H_i})(X_{-j})&=(\mathrm{Id} \otimes \partial_{H_i} + \partial_{H_i} \otimes \mathrm{Id}) \circ \Delta(X_{-j})
\end{align*}

and $\Delta \circ \partial_{H_i} (K_{\gamma}) = 0 = (\partial_{H_i} \otimes 1 + 1 \otimes \partial_{H_i}) \circ \Delta(K_{\gamma})$, so $\mfh$ acts on $U_{q}(\mfg)$ by $\mbbC$-biderivations. It follows from \cite[Lemma 2.6]{AS} that $U_{q}(\mfg) \rtimes U(\mfh):= U_{q}(\mfg) \otimes U(\mfh)$ is a Hopf algebra with algebra structure coming from the smash product (the unit here is $1\otimes 1$) and coalgebra structure coming from the tensor product:
\begin{align*}
(X \otimes H) \cdot (Y \otimes H')&=X(\partial_{H_{(1)}}Y) \otimes H_{(2)}H'\\
\Delta(X\otimes H)&=(\mathrm{Id} \otimes \tau \otimes \mathrm{Id}) \circ (\Delta_{U_{q}(\mfg)} \otimes \Delta_{U(\mfh)}) (X \otimes H)\\
\epsilon&=\epsilon_{U_{q}(\mfg)} \otimes \epsilon_{U(\mfh)}
\end{align*}
where $\tau:U_{q}(\mfg) \otimes U(\mfh) \to U(\mfh) \otimes U_{q}(\mfg)$ is the usual flip map and $\Delta(H)=\sum\limits_{(H)} H_{(1)} \otimes H_{(2)}$ is the Sweedler notation for the coproduct. We then see that for any $X \in U_{q}(\mfg)$ and $H_i \in U(\mfh)$, we have
\begin{align*}
 (X \otimes 1) \cdot (1 \otimes H_i)&=X \otimes H_i & (X \otimes 1) \cdot (Y \otimes 1)&=XY \otimes 1\\
(1 \otimes H_i) \cdot (X \otimes 1)&=\partial_{H_i}(X) \otimes 1+X \otimes H_i &(1 \otimes H_i) \cdot (1 \otimes H_j) &= 1 \otimes H_iH_j,\\
\end{align*}
so $U_{q}(\mfg) \rtimes U(\mfh)$ is generated by the elements $\{ X_{\pm i} \otimes 1, K_{\gamma} \otimes 1, 1 \otimes H_i \, | \, i=1,...,n, \gamma \in L\}.$ By abuse of notation, we set $X_{\pm i}:=(X_{\pm_i} \otimes 1), K_{\gamma}:=(K_{\gamma} \otimes 1),$ and $H_i:=(1\otimes H_i)$, and we see that $U_{q}(\mfg) \rtimes U(\mfh)$ is generated by $X_{\pm i}, K_{\gamma}$ $(\gamma \in L)$, and $H_i$, $i=1,...,n$ with defining relations \eqref{eqKX}-\eqref{serre} and Hopf algebra structure given by equations \eqref{coK}-\eqref{coH}. This is precisely the unrolled quantum group of Definition \ref{uqh}.

\subsection{Braid Group Action on $U_{q}^H(\mfg)$}\label{subsectionPBW}

Recall that for a finite dimensional semisimple Lie algebra the scalars $a_{ij}a_{ji}$ are equal to $0,1,2,$ or $3$ for $i \not = j$ and for each case let $m_{ij}$ be $2,3,4,6$ respectively. Then,
\begin{definition}\label{dbraid}
The braid group $\mcB_{\mfg}$ associated to $\mfg$ has generators $T_i$ with $1 \leq i \leq n$ and defining relations
\[ T_iT_jT_i \cdots = T_jT_iT_j \cdots \]
for $i \not = j$ where each side of the equation is a product of $m_{ij}$ generators.
\end{definition}

It is well known (see \cite{KlS,CP}) that the braid group of $\mfg$ acts on the quantum group $U_{q}(\mfg)$ by algebra automorphisms defined as follows:
\begin{equation} \label{b1}T_i(K_j)=K_jK_i^{-a_{ij}}, \qquad T_i(X_i)=-X_{-i}K_i, \qquad T_i(X_{-i})=-K_i^{-1}X_i
\end{equation}
\begin{align}
\label{b2}T_i(X_j)&=\sum\limits_{t=0}^{-a_{ij}}(-1)^{t-a_{ij}}q_i^{-t}X_i^{(-a_{ij}-t)}X_jX_i^{(t)}\quad i \not = j,\\
\label{b3}T_i(X_{-j})&=\sum\limits_{t=0}^{-a_{ij}} (-1)^{t-a_{ij}}q_i^tX_{-i}^{(t)}X_{-j}X_{-i}^{(-a_{ij}-t)}, \quad i \not = j, 
\end{align}
where $X_{\pm i}^{(n)}=X_{\pm i}^n/[n]_{q_i}!$. Therefore, if we extend the action of $\mcB_{\mfg}$ to $U_{q}^H(\mfg)$ by
\[ T_i(H_j)=H_j-a_{ji}H_i, \]
it is enough to check that the automorphisms $T_i$, $i=1,...,n$ respect equations \eqref{eqH} and the braid group relations when acting on the $H_i$. We first note that the relations
\[ [T_k(H_i),T_k(H_j)]=[T_k(H_i),T_k(K_{\gamma})]=0\]
follow trivially from $[H_i,H_j]=[H_i,K_{\gamma}]=0$. We must therefore show that
\[ [T_k(H_i),T_k(X_{\pm j } )]=\pm a_{ij}T_k(X_{ \pm j}).\]
We will prove the statement for positive index, as the negative index case is identical. Suppose $k=j$, then $T_j(X_j)=-X_{-j}K_j$ and we have 
\[ H_k(-X_{-j}K_j)=-(X_{-j}H_k-a_{kj}X_{-j})K_j=-X_{-j}K_jH_k+a_{kj}X_{-j}K_j, \]
so we see that $[H_k,-X_{-j}K_j]=a_{kj}X_{-j}K_j$. Hence,
\begin{align*}
[T_j(H_i),T_j(X_j)]&=[H_i-a_{ij}H_j,-X_{-j}K_j]\\
&=[H_i,-X_{-j}K_j]-a_{ij}[H_j,-X_{-j}K_j]\\
&=a_{ij}X_{-j}K_j-2a_{ij}X_{-j}K_j\\
&=-a_{ij}X_{-j}K_j=a_{ij}T_j(X_j).
\end{align*}
Suppose now that $k \not = j$. We then see that we must show
\[ \sum\limits_{t=0}^{-a_{kj}} (-1)^{t-a_{kj}}q_k^{-t}[T_k(H_i),X_k^{(-a_{kj}-t)}X_jX_k^{(t)}]=\sum\limits_{t=0}^{-a_{kj}} (-1)^{t-a_{kj}}q_k^{-t}a_{ij}X_k^{(-a_{kj}-t)}X_jX_k^{(t)}.\]
Clearly then, it is enough to show that
\[ [T_k(H_i),X_k^{(-a_{kj}-t)}X_jX_k^{(t)}]=a_{ij}X_k^{(-a_{kj}-t)}X_jX_k^{(t)}\]
for each $t=0,...,-a_{kj}$. It follows easily from equation \eqref{eqH} that
\[ [H_i,X_k^{(-a_{kj}-t)}X_jX_k^{(t)}]=(-a_{kj}a_{ik}+a_{ij})X_k^{(-a_{kj}-t)}X_jX_k^{(t)}.\]
Therefore,
\begin{align*}
[T_k(H_i),X_k^{(-a_{kj}-t)}X_jX_k^{(t)}]&=[H_i,X_k^{(-a_{kj}-t)}X_jX_k^{(t)}]-a_{ik}[H_k,X_k^{(-a_{kj}-t)}X_jX_k^{(t)}]\\
&=(-a_{kj}a_{ik}+a_{ij}-a_{ik}(-2a_{kj}+a_{kj}))X_k^{(-a_{kj}-t)}X_jX_k^{(t)}\\
&=a_{ij}X_k^{(-a_{kj}-t)}X_jX_k^{(t)}.
\end{align*}
We see then that equations \eqref{b1}-\eqref{b4} define an automorphism of $U_{q}^H(\mfg)$. To show that these automorphisms give an action of the braid group, we need only show that they satisfy the braid relations as operators on $U_{q}^H(\mfg)$. We know these relations are satisfied for the elements of $U_{q}(\mfg)$, so we need only check the $H_i$. This amounts to showing that
\[ T_iT_j \cdots (H_k) = T_jT_i \cdots (H_k). \]
One therefore computes $T_iT_j \cdots (H_k)$ and checks that the result is symmetric in $i$ and $j$, giving the following proposition:

\begin{proposition}\label{braid}
The elements $T_i$ of the braid group $\mcB_{\mfg}$ act on $U_{q}^H(\mfg)$ by automorphisms given by relations \eqref{b1}-\eqref{b3} and
\begin{equation}
\label{b4} T_i(H_j)=H_j-a_{ji}H_i.
\end{equation}
\end{proposition}

For the remainder of the article we assume that $q$ is an $\ell$-th root of unity with 
\begin{equation} \label{rl} r=\ell \text{ if $\ell$ is odd}, \qquad \qquad  r=\ell/2 \text{ if $\ell$ is even} \end{equation}
Let $W$ denote the Weyl group of $\mfg$ and $\{s_i|i=1,...,n\}$ the simple reflections generating $W$. Let $s_{i_1} \cdots s_{i_N}$ be a reduced decomposition of the longest element $\omega_0$ of $W$. Then, $\beta_k:=s_{i_1}s_{i_2} \cdots s_{i_{k-1}}\alpha_{i_k}$, $k=1,...,N$ gives a total ordering on the set of positive roots $\Delta^+$ of $\mfg$ and for each $\beta_i$, $i=1,...,N$, we can associate the root vectors $X_{\pm \beta_i} \in \Uq{\mfg}$ as seen in \cite[Subsection 9.1]{CP}. We have the following PBW theorem \cite[Proposition 9.2.2]{CP}:
 
 \begin{theorem}\label{PBW}
 The multiplication operation in $\Uq{\mfg}$ defines a vector space isomorphism
 \[ U_q(\eta^-) \otimes U_q(\mfh) \otimes U_q(\eta^+) \cong \Uq{\mfg} \]
 
where $U_q(\eta^{\pm})$ is the subalgebra generated by the $X_{\pm \alpha_i}$ and $U_q(\mfh)$ the subalgebra generated by the $K_\gamma$. The set $\{X_{\pm \beta_1}^{k_1}X_{\pm \beta_2}^{k_2} \cdots X_{\pm \beta_N}^{k_N}| k_i \in \mathbb{Z}_{\geq0}\}$ is a basis of $U_q(\eta^{\pm})$.
 \end{theorem}

It can be shown by induction on $s$ \cite{J} that 
\begin{equation}\label{eqJ} [X_i,X_{-i}^s]=[s]_iX_{-i}^{s-1}[K_i;d_i(1-s)], \end{equation}
where $[K_i;n]=(K_iq^n-K_i^{-1}q^{-n})/(q_i-q_i^{-1})$. Let $d_{\alpha}=\frac{1}{2}\langle \alpha, \alpha \rangle$, and define 
\begin{equation}\label{eqr}r_{\alpha}:=r/\mathrm{gcd}(d_{\alpha},r).\end{equation}
where we recall the definition of $r$ in Equation \eqref{rl}. Then $[r_{\alpha_i}]_i=[r]=0$, so it follows from equations \eqref{eqX} and \eqref{eqJ} that $[X_j,X_{-i}^{r_{\alpha_i}}]=0$ for all $i,j$. Applying the braid group action then gives $[X_{\beta},X_{-\alpha}^{r_{\alpha}}]=0$ for all $\alpha,\beta \in \Delta^{+}$, where we have used the fact that $d_{\alpha}=d_i$ when $\alpha$ lies in the Weyl orbit of $\alpha_i$. It follows that given any maximal vector $v$ (i.e. $X_{i}v=0$ for all $i$) in some $U_q^H(\mfg)$-module $V$, $X_{-\alpha}^{r_{\alpha}}v$ is also maximal. In particular, all Verma modules of $U_q^H(\mfg)$ will be reducible and therefore there will be no projective irreducible modules (as all irreducibles are quotients of Verma modules). Obtaining a category of representations which is generically semisimple is the motivation for our choice of definition of the unrolled restricted quantum group at arbitrary roots and to do this, we quotient out $\{X_{\pm\alpha}^{r_{\alpha}}\}_{\alpha \in \Delta^+}$. It follows from equation \eqref{eqKX} that $X_i \otimes K_i X_i=q^{2d_i}X_i \otimes X_iK_i$, so equation \eqref{coXi} and the $q$-binomial formula tell us that
\[ \Delta(X_i^{r_{\alpha_i}})= \sum\limits_{k=0}^{r_{\alpha_i}} \binom{r_{\alpha_i}}{k}_{q^{2d_i}} X_i^n\otimes K_i^n X_i^{r_{\alpha_i}-n} = 1 \otimes X_i^{r_{\alpha_i}} + X_i^{r_{\alpha_i}} \otimes K_i^{r_{\alpha_i}}\]
since $\binom{r_{\alpha_i}}{k}=1$ if $k=0,r_{\alpha_i}$ and zero otherwise. We can perform the same computation for $X_{-i}^{r_{\alpha_i}}$, so we see that the two-sided ideal generated by $\{X_{\pm \alpha_i}^{r_{\alpha_i}}\}_{\alpha_i \in \Delta}$ is a Hopf ideal (it follows immediately from equations \eqref{eqKX} and \eqref{coXi} that this ideal is invariant under the antipode $S$).
\begin{definition}\label{restricteddef}
The unrolled restricted quantum group of $\mfg$, $\UHbar{\mfg}$, is defined to be the unrolled quantum group $\UH{\mfg}$ of Definition \ref{uqh} quotiented by the Hopf ideal generated by $\{X_{\pm \alpha_i}^{r_{\alpha_i}} \}_{\alpha_i \in \Delta}$.
\end{definition}
This definition is very closely related to that of the small quantum group in \cite{L}. It follows trivially from the braid relations that $X_{\pm \alpha}^{r_{\alpha}}=0$ in $\overline{U}_q^H(\mfg)$ for every root vector $\alpha \in \Delta^+$, and it follows from the PBW theorem that $\{X_{\pm \beta_1}^{k_1}X_{\pm \beta_2}^{k_2} \cdots X_{\pm \beta_N}^{k_N} \;| \; 0 \leq k_i < r_{\beta_i}\}$ is a basis of $\overline{U}_q^H(\eta^{\pm})$. There exists an automorphism of $U_q(\mfg)$ which swaps $X_i$ with $X_{-i}$ and inverts $K_i$ (see \cite[Lemma 4.6]{J}). It is easily checked that this automorphism can be extended to $\UH{\mfg}$ by defining $\omega(H_i)=-H_i$, so there exists an automorphism $\omega: \UH{\mfg} \to \UH{\mfg}$ defined by
\begin{equation}\label{eqw} \omega(X_{\pm i})=X_{\mp i}, \qquad \omega(K_{\gamma})=K_{-\gamma},\qquad \omega (H_i)=-H_i.
\end{equation}
This automorphism will appear in Section \ref{Repsection} in the definition of a Hermitian form on Verma modules introduced in \cite{DCK}, and the definition of a contravariant functor analogous to the duality functor for Lie algebras.

\section{\textbf{Representation Theory of $\UHbar{\mfg}$}}\label{Repsection}
 
For each module $V$ of $\UHbar{\mfg}$ and $\lambda \in \cartan^*$, define the set $V(\lambda):=\{ v \in V \; | \; H_iv=\lambda(H_i)v\}$. If $V(\lambda)\not = 0$, then we call $\lambda$ a weight of $V$, $V(\lambda)$ its weight space, and any $v \in V(\lambda)$ a weight vector of weight $\lambda$. 
\begin{definition}\label{weightmod}A $\overline{U}_q^H(\mfg)$-module $V$ is called a weight module if $V$ splits as a direct sum of weight spaces and for each $\gamma=\sum\limits_{i=1}^n k_i \alpha_i \in L$, $K_{\gamma}=\prod\limits_{i=1}^nq_i^{k_iH_i}$ as operators on $V$. We define $\mcC$ to be the category of finite dimensional weight modules for $\UHbar{\mfg}$. \end{definition}
Given any $V \in \mcC$, we denote by $\Gamma(V)$ the set of weights of $V$. That is,
\begin{align}\label{weightset}
\Gamma(V)=\{\lambda \in \mfh^*\; | \; V(\lambda) \not = 0\}.
\end{align} 
We define the character of a module $V \in \mcC$ using the dimensions of the $H$-eigenspaces as
\begin{equation}\label{character}
\mathrm{ch}[V]=\sum\limits_{\lambda \in \mfh^*} \mathrm{dim}V(\lambda)z^{\lambda}.
\end{equation}
It is easy to show that for any module $V$ and $\lambda \in \cartan^*$,
\begin{align}\label{rootlattice} X_{\pm j}V(\lambda) \subset V(\lambda \pm \alpha_j). 
\end{align}
We define the usual partial order ``$\geq$" on $\cartan^*$ by $\lambda_1 \geq \lambda_2$ iff $\lambda_1=\lambda_2 +\sum\limits_{i=1}^n k_i \alpha_i$ for some $k_i \in \mathbb{Z}_{\geq 0}$. A weight $\lambda$ of $V$ is said to be highest weight if it is maximal with respect to the partial order among the weights of $V$. A vector $v \in V$ is called maximal if $X_iv=0$ for each $i$, and a module generated by a maximal vector will be called highest weight. 

Given a weight $\lambda \in \mfh^*$, denote by $I^{\lambda}$ the ideal of $\UHbar{\mfg}$ generated by the relations $H_i\cdot 1=\lambda(H_i)$, $K_{\gamma} \cdot 1=\prod\limits_{i=1}^nq_i^{k_i \lambda(H_i)}$ for $\gamma=\sum\limits_{i=1}^n k_i\alpha_i \in L$ and $X_i \cdot 1=0$ for each $i$.

\begin{definition} \label{Verma}
Define $ \Verma{\lambda}:=\UHbar{\mfg} / I^{\lambda}$. $\Verma{\lambda}$ is generated as a module by the coset $v_{\lambda}:=1+I^{\lambda}$ with relations
\[ X_iv_{\lambda}=0, \qquad H_iv_{\lambda}=\lambda(H_i)v_{\lambda},  \qquad K_{\gamma}v_{\lambda}=\prod\limits_{i=1}^nq_i^{k_i\lambda(H_i)}v_{\lambda}, \]
where $\gamma=\sum\limits k_i \alpha_i \in L$. It follows from Theorem \ref{PBW} that $M^{\lambda}$ has basis $\{X_{-\beta_1}^{k_1}X_{-\beta_2}^{k_2} \cdots X_{-\beta_N}^{k_N}v_{\lambda} \; | \; 0\leq k_i < r_{\beta_i} \}$.
\end{definition}

Clearly, $\Verma{\lambda} \in \mcC$ and is universal with respect to highest weight modules in $\mcC$, that is, for any module $M \in \mcC$ generated by a highest weight vector of weight $\lambda$, there exists a surjection $M^{\lambda} \twoheadrightarrow M$. Each proper submodule $S$ of $\Verma{\lambda}$ is a direct sum of its weight spaces and has $S(\lambda)=\emptyset$, so the union of all proper submodules is a maximal proper submodule. Hence, each reducible $\Verma{\lambda}$ has a unique maximal proper submodule $S^{\lambda}$ and unique irreducible quotient $L^{\lambda}$ of highest weight $\lambda$. We therefore refer to $\Verma{\lambda}$ as the Verma (or universal highest weight) module of highest weight $\lambda$ and we have the following proposition by standard arguments:
\begin{proposition}
$V \in \mcC$ is irreducible iff $V \cong L^{\lambda}$ for some $\lambda \in \mfh^*$.
\end{proposition}

It is clear that every module in $\mcC$ is a module over $U_q(\mfg)$ and since the $H_i$ act semi-simply, $\Verma{\lambda}$ is irreducible iff it is irreducible as a $U_q(\mfg)$-module. Kac and De Concini defined a Hermitian form $H$ on $M^{\lambda}$ \cite[Equation 1.9.2]{DCK} by
\[H(v_{\lambda},v_{\lambda})=1 \quad \mathrm{and} \quad H(Xu,v)=H(u, \omega(X)v)\]
for all $X \in \overline{U}_q^H(\mfg)$ and $u,v \in M^{\lambda}$ where $\omega$ is the automorphism defined in equation \eqref{eqw}. Let $\eta \in \Delta^+$ and denote by $\mathrm{det}_{\eta}(\lambda)$ the determinant of the Gram matrix of $H$ restricted to $M^{\lambda}(\lambda-\eta)$ in the basis consisting of elements $F_{\beta_1}^{k_1} \cdots F_{\beta_N}^{k_N} v_{\lambda}$ with $\vec{k}=(k_1,...,k_N) \in \mathrm{Par}(\eta):=\{ \vec{k} \in \mbbZ^N \, | \, \sum k_i \beta_i=\eta, \, 0 \leq k_i < r_{\beta_i} \}$. The determinant of $H$ vanishes precisely on the maximal submodule of $M^{\lambda}$ and is given on $M^{\lambda}(\lambda-\eta)$ by \cite[Equation 1.9.3]{DCK}
\begin{align*} \mathrm{det}_{\eta}(\lambda)&=\prod\limits_{\alpha \in \Delta^+}\prod\limits_{m =0}^{r_{\alpha}-1} \left(\frac{ \{md_{\alpha}\}}{\{d_{\alpha}\}^2} \right)^{|\mathrm{Par}(\eta - m\alpha)|} (\lambda(K_{\alpha})q^{\langle \rho, \alpha \rangle-\frac{m}{2}\langle \alpha , \alpha \rangle}-\lambda(K_{\alpha}^{-1})q^{-\langle \rho, \alpha \rangle +\frac{m}{2} \langle \alpha , \alpha \rangle})^{| \mathrm{Par}(\eta-m\alpha)|}
\end{align*}
where $\rho$ is the Weyl vector and by abuse of notation we denote by $\lambda(K_{\alpha})$ the scalar by which $K_{\alpha}$ acts on the highest weight vector of $M^{\lambda}$. It then follows as in \cite[Theorem 3.2]{DCK} that we have the following:
\begin{proposition}\label{irreducible}
$\Verma{\lambda}$ is irreducible iff $q^{2\langle \lambda +\rho,\alpha \rangle-k\langle \alpha, \alpha \rangle} \not = 1$ for all $\alpha \in \Delta^+$ and $k=1,...,r_{\alpha}-1$.
\end{proposition}
For each $\alpha \in \Delta^+$ we associate to $\lambda$ the scalars $\lambda_{\alpha}\in \mbbC$ defined by 
\[ \lambda_{\alpha}:=\langle \lambda+\rho, \alpha \rangle.\] 
Notice that $\Verma{\lambda}$ is reducible iff for some $\alpha \in \Delta^+$ we have
\begin{equation} \label{kn}
2(\lambda_{\alpha}-k_{\alpha}^{\lambda} d_{\alpha})=n_{\alpha}^{\lambda} \ell
\end{equation}
for some $k_{\alpha}^{\lambda} \in \{1,...,r_{\alpha}-1\}$, $n_{\alpha}^{\lambda} \in \mbbZ$, where $d_{\alpha}:=\frac{1}{2}\langle \alpha, \alpha \rangle$. This motivates the following definition:
\begin{definition}\label{typicaldef}
We call the scalar $\lambda_{\alpha}$ typical if $2(\lambda_{\alpha} - kd_{\alpha}) \not = 0 \; \mathrm{mod \, \ell}$ for all $k=1,...,r_{\alpha}-1$ and atypical otherwise. We call $\lambda \in \mfh^*$ typical if $\lambda_{\alpha}$ is typical for all $\alpha \in \Delta^+$ and atypical otherwise.
\end{definition}
Clearly, $\Verma{\lambda}$ is irreducible iff $\lambda$ is typical. We can rewrite the atypicality condition into a more convenient form, which will be useful in the next subsection:
\begin{proposition}\label{typ}
$\lambda_{\alpha}$ is typical iff $\lambda_{\alpha} \in \ddot{\mathbb{C}}_{\alpha}$ where
\[ \ddot{\mathbb{C}}_{\alpha}:=\begin{cases}
(\mbbC \setminus g_{\alpha}\mathbb{Z}) \cup r\mathbb{Z} & \text{ if $\ell$ is even}\\
(\mbbC \setminus \frac{g_{\alpha}}{2}\mathbb{Z}) \cup \frac{r}{2}\mathbb{Z} & \text{ if $\ell$ is odd}
\end{cases}
 \]
where $g_{\alpha}=gcd(d_{\alpha},r)$.
\end{proposition}
\begin{proof} By Proposition \ref{irreducible} and the following comments, $\lambda_{\alpha}$ is atypical iff $2(\lambda_{\alpha}-kd_{\alpha}) =0 \; \mathrm{mod \, \ell}$ for some $k=1,...,r_{\alpha}-1$. That is, iff
\[ \lambda_{\alpha} \in \bigcup\limits_{n \in \mathbb{Z}} \bigcup\limits_{k =1}^{r_{\alpha}-1}\frac{ n\ell+2kd_{\alpha}}{2} \]

Assume now that $gcd(d_i,r)=1$ for all $i$. Note that each non-simple root $\alpha$ lies in the Weyl orbit of some simple root $\alpha_i$ and that $d_{\alpha}=d_{\alpha_i}$, so $gcd(d_{\alpha},r)=1$ for all $\alpha \in \Delta^+$ (so $r_{\alpha}=r$ for all $\alpha$). Let $r=\ell$ be odd, then we have $\lambda_{\alpha} \in \bigcup\limits_{n \in \mathbb{Z}} \bigcup\limits_{k =1}^{r-1}\frac{nr+2kd_{\alpha}}{2}$ which is clearly a subset of $\frac{1}{2}\mathbb{Z}$. Suppose $nr+2kd_{\alpha}=rm$ for some $m \in \mathbb{Z}$. Then $r(m-n)=2kd_{\alpha}$ and we have $gcd(r,2d_{\alpha})=1$ so we must have $2d_{\alpha} | m-n$. However, we have $k \in \{1,...,r-1\}$ so $|r(m-n)|>|2kd_{\alpha}|$, a contradiction. Hence, $\lambda_{\alpha} \in \frac{1}{2}\mathbb{Z} \setminus \frac{r}{2}\mathbb{Z}$. Let $x \in \mathbb{Z} \setminus r\mathbb{Z}$. Since $gcd(r,2d_{\alpha})=1$, there exist $a,b \in \mathbb{Z}$ such that $2d_{\alpha}a+br=1$. Since $x \not \in r\mathbb{Z}$, we have $ax \not \in r\mathbb{Z}$, otherwise $x=2d_{\alpha}ax+brx \in  r\mathbb{Z}$. Therefore, there exist $m, k \in \mathbb{Z}$ with $k=1,...,r-1$ such that $ax=mr+k$. Then, $2d_{\alpha}ax=2d_{\alpha}mr+2d_{\alpha}k$, so $x=2d_{\alpha}k \; \mathrm{mod \, r}$ (since $2d_{\alpha}a=1 \; \mathrm{mod \, r}$) and so $x \in \bigcup\limits_{n \in \mathbb{Z}} \bigcup\limits_{k=1}^{r-1} rn+2d_{\alpha}k$. Hence, we have shown
\[ \bigcup\limits_{n \in \mathbb{Z}} \bigcup\limits_{k =1}^{r-1}\frac{nr+2d_{\alpha}k}{2}= \frac{1}{2}\mathbb{Z} \setminus \frac{r}{2}\mathbb{Z}\]

A similar argument shows that $\bigcup\limits_{n \in \mathbb{Z}} \bigcup\limits_{k =1}^{r-1}\frac{ n\ell+2kd_{\alpha}}{2}=\mathbb{Z} \setminus r\mathbb{Z}$ when $\ell$ is even. Suppose now that $gcd(d_{\alpha},r) \not = 0$. Then we have 
\[ \bigcup\limits_{n \in \mathbb{Z}} \bigcup\limits_{k =1}^{r_{\alpha}-1}\frac{(nr_{\alpha}+2k)d_{\alpha}}{2}=\begin{cases} d_{\alpha} \mathbb{Z} \setminus r \mathbb{Z} & \ell \mathrm{ \; even}\\
 \frac{d_{\alpha}}{2}\mathbb{Z} \setminus \frac{r}{2} \mathbb{Z} & \ell \mathrm{\; odd}
\end{cases} \]
\end{proof}

\begin{remark}
The invertible objects in $\mcC$ are clearly the $1$-dimensional $L^{\lambda}$. Note that we have
\[ X_jX_{-i}v_{\lambda}=\delta_{i,j}[\lambda(H_i)]_iv_{\lambda}\]
so $L^{\lambda}$ is $1$-dimensional iff $2\lambda(H_i)d_i = 0 \mathrm{ \; mod\;} \ell$ i.e. $\lambda(H_i) \in \frac{\ell}{2d_i} \mathbb{Z}$ for all $i$. These objects played a crucial role in \cite{CGR} for the construction of certain quasi-Hopf algebras $\overline{U}_q^{\Phi}(\mathfrak{sl}_2)$ whose representation theory related to the triplet VOA. We expect this to remain true in the higher rank case, and will be investigated in future work. We also expect the higher rank analogues of $\overline{U}_q^{\Phi}(\mathfrak{sl}_2)$ to be closely related to those quantum groups which appear in \cite{N,GLO}.
\end{remark}

\subsection{Categorical Structure} \label{sectioncategoricaldata}
We first remark that it is easily seen (as in \cite[Subsection 5.6]{GP1}), that the square of the antipode acts as conjugation by $K_{2\rho}^{1-r}$, i.e.
\[ S^2(x)=K_{2\rho}^{1-r}xK_{2\rho}^{r-1} \]
for $x \in \UHbar{\mfg}$ where $\rho:= \frac{1}{2}\sum\limits_{\alpha \in \Delta^+} \alpha$ is the Weyl vector. A Hopf algebra in which the square of the antipode acts as conjugation by a group-like element is pivotal (see \cite[Proposition 2.9]{B}), so $\mathcal{C}$ is pivotal. It is clear from equation \eqref{rootlattice} that given any $v \in V(\lambda)$, every weight for the submodule $\langle v\rangle \subset V$ generated by $v$ has the form $\lambda+\sum\limits_{i=1}^n k_i\alpha_i$ for some $k_i \in \mathbb{Z}$. That is, every weight vector in $\langle v \rangle$ has weight differing from $\lambda$ by an element of the root lattice $L_R$ of $\mfg$. We can quotient $\mfh^*$ by $L_R$ to obtain the group $\mft:=\mfh^*/L_R$. We then define $\mcC_{\overline{\lambda}}$ for $\overline{\lambda} \in \mft$ to be the full subcategory of $\mcC$ consisting of modules whose weights differ from $\lambda$ by an element of $L_R$. Clearly, 
\begin{align}\label{grading} 
\mcC= \bigoplus\limits_{\overline{\lambda} \in \mft} \mcC_{\overline{\lambda}}\end{align}
and it is easy to see that this gives $\mcC$ a $\mft$-grading as in Section \ref{prelim}. Note that all $\mathcal{C}_{\overline{\lambda}}$ are non-zero since any representative $\lambda \in \mfh^*$ of $\overline{\lambda}$ has a simple highest weight module $L^{\lambda} \in \mathcal{C}_{\overline{\lambda}}$. In fact, we have the following:

\begin{proposition}\label{gens}
$\mcC$ is generically $\mft$-semisimple.
\end{proposition}
\begin{proof}
Let $\mcX$ be the subset of $\mft$ consisting of equivalence classes $\overline{\lambda}$ corresponding to weights $\lambda$ such that $\lambda_{\alpha} \in  \frac{(3+(-1)^{\ell})g_{\alpha}}{4}\mathbb{Z}$ for some $\alpha \in \Delta^+$. Notice that this implies that $\mu_{\alpha} \in  \frac{(3+(-1)^{\ell})g_{\alpha}}{4}\mathbb{Z}$ for all $\mu \in \overline{\lambda}$ since $\mu$ being comparable to $\lambda$ implies that
\begin{equation} \label{eqd} \mu_{\alpha}=\lambda_{\alpha} \; \mathrm{mod} \; d_{\alpha} \mbbZ. \end{equation}
Equation \eqref{eqd} is easy to see for simple roots and for non-simple roots, one uses invariance of $\langle -,- \rangle$ under the action of the Weyl group. $\mcX$ is clearly symmetric. To see $\mathcal{X}$ is small, consider the subset 
\[ A:= \{ \overline{\mu^a} \in \mft \; | \;  (\mu^a)_{\alpha}=ai \text{ for some $a\not = 0 \in \mathbb{R}$ and all $\alpha \in \Delta^+$}\} \subset \mft \]
where $i=\sqrt{-1}$. Each $\overline{\mu^a}$ is distinct since the corresponding weights do not differ by elements of the root lattice. Suppose that $\overline{\mu^a} +\mcX=\overline{\mu^b} + \mcX$ for some $a \not = b$, both non-zero. Then $\overline{\mu^a}-\overline{\mu^b}\in \mcX$, a contradiction since any element $\mu \in \overline{\mu^a}-\overline{\mu^b}$ will have purely imaginary $\mu_{\alpha}$ for each $\alpha$, so $\mu$ cannot belong to $\mcX$. Therefore, $A$ cannot be covered by finitely many translations of $\mcX$ and $\mcX$ is small. Notice that by construction, $\mft \setminus \mcX$ consists of equivalence classes $\overline{\lambda}$ such that every weight of every module $V \in \mcC_{\overline{\lambda}}$ is typical.

The argument in \cite[Lemma 7.1]{CGP2} can be applied to our setting to show that the irreducibles of typical weight are projective. Recall that $\{X_{\pm \beta_1}^{k_1}X_{\pm \beta_2}^{k_2} \cdots X_{\pm \beta_N}^{k_N} \;| \; 0 \leq k_i < r_{\beta_i} \}$ is a basis of $\overline{U}_q^H(\eta^{\pm})$. Let $X_+:=\prod\limits_{k=1}^N X_{\beta_k}^{r_{\beta_k}-1}$ and $X_{-}:=\prod\limits_{k=1}^N X_{-\beta_k}^{r_{\beta_k}-1}$ denote the highest and lowest weight vectors of $\overline{U}_q^H(\mfg)$. Suppose $\lambda \in \mfh^*$ is of typical weight and that there is a surjection $f:M \to M^{\lambda}$ for some $M \in \mcC$. $M^{\lambda}$ is a Verma module so clearly $X_-v_{\lambda} \not = 0$ where $v_{\lambda}$ is a highest weight vector of $M^{\lambda}$. Since $M^{\lambda}$ is irreducible, there exists an element $X \in \overline{U}_q^H(\mathfrak{g})$ such that $XX_-v_{\lambda}=\nu v_{\lambda}$ for some non-zero scalar $\nu \in \mathbb{C}$. However, the only elements $X \in \overline{U}_q^H(\mathfrak{g})$ such that $XX_-v_{\lambda}$ has weight $\lambda$ are scalar multiples of $X_+$. Therefore, $X_+X_-v_{\lambda}=\nu v_{\lambda}$ for some non-zero $\nu \in \mathbb{C}$. Then there is a vector $w \in f^{-1}(\frac{1}{\nu} v_{\lambda})$. The vector $w':=X_+X_-w$ is maximal since $X_+$ is maximal in $\overline{U}_q^H(\eta^+)$ (i.e. $XX_+=0$ for all $X \in \overline{U}_q^H(\eta^+)$) and non-zero since $f(w')=v_{\lambda}$. Therefore, by the universal property of Verma modules there is a map $g:M^{\lambda} \to M$ such that $g(v_{\lambda})=w'$ and $f:M \to M^{\lambda}$ splits. Hence, every $\mcC_{\overline{\lambda}}$ with $\overline{\lambda} \in \mft \setminus \mcX$ contains only projective irreducible modules and is therefore semisimple. 
\end{proof}

It is well known (see \cite[Section 5.6]{GP1}, for example) that the duality morphisms are given by
{\allowdisplaybreaks
\begin{align*}
\overrightarrow{\mathrm{coev}}_V&: \mathds{1} \rightarrow V \otimes V^*, \qquad 1\mapsto \sum\limits_{i \in I} v_i \otimes v_i^*,\\
\overrightarrow{\mathrm{ev}}_V&:V^* \otimes V \rightarrow \mathds{1}, \qquad f \otimes v \mapsto f(v),\\
\overleftarrow{\mathrm{coev}}_V&:\mathds{1} \rightarrow V^* \otimes V, \qquad 1 \mapsto \sum\limits_{i \in I} v_i^* \otimes K_{2\rho}^{r-1}v_i,\\
\overleftarrow{\mathrm{ev}}_V&: V \otimes V^* \rightarrow \mathds{1}, \qquad v \otimes f \mapsto f(K_{2\rho}^{1-r}v),\\
\end{align*}
where $\mathds{1}$ is the $1$-dimensional module of weight zero and $\{v_i\}_{i \in I}$, $\{v_i^*\}_{i \in I}$ are dual bases of $V$ and $V^*$. The pivotal structure on $\mcC$ is the monoidal natural transformation $\delta:\mathrm{Id}_{\mcC} \to (-)^{**}$ defined by components 
\[\delta_V=\psi_V \circ \ell_{K_{2\rho}^{1-r}}:V \to V^{**}\]
 where $\psi_V:V \to V^{**}$ is the canonical embedding $\psi_V(v)(f)=f(v)$ and $\ell_x(v)=xv$ denotes left multiplication. It was shown in \cite[Subsection 5.8]{GP1} (see also \cite[Subsection 4.2]{GP2}) that the category of finite dimensional weight modules for the unrolled restricted quantum group $\UHbar{\mfg}$ is braided.

The proof of this statement is given for odd roots, but holds for even roots as well with very minor adjustments. The proof uses a projection map $\bar{p}:U_h(\mfg) \to U^{<}$ from the $h$-adic quantum group $U_h(\mfg)$ to the $\mbbC[[h]]$-module generated by the monomials 
\begin{equation}\label{eqmon} \prod\limits_{i=1}^n H_{i}^{m_i} \prod\limits_{j_1=1}^N X_{\beta_{j_1}}^{k_{j_1}} \prod\limits_{j_2=1}^N X_{-\beta_{j_2}}^{k_{j_2 }}\end{equation}
with $m_i \in \mbbZ_{\geq 0}$, $0 \leq k_{j_1},k_{j_2} <r$. If we generalize this by defining $U^<$ to be the $\mbbC[[h]]$-module generated by the monomials in equation \eqref{eqmon} with $m_i \in \mbbZ_{\geq 0}$ and $0 \leq k_{j_s}< r_{\beta_{j_s}}$, then the proof follows verbatim as in \cite[Subsection 5.8]{GP1}. This yields an $R$-matrix $R^h:=\mathcal{H}\tilde{R}$ where 
\begin{align}
\mathcal{H}&:=q^{\sum\limits_{i,j}d_i(A^{-1})_{ij}H_i \otimes H_j},\\
\tilde{R}&:=\prod\limits_{i=1}^N\left( \sum\limits_{j=0}^{r_{\beta_i}-1} \frac{\left( (q_{\beta_i}-q_{\beta_i}^{-1})X_{\beta_i} \otimes X_{-\beta_i}\right)^j}{[j;q_{\beta_i}^{-2}]!}\right),
\end{align}
where $\{\beta_i\}_{i=1}^{N}$ is the ordered bases for $\Delta^+$ as described in Subsection \ref{subsectionPBW}, $[j;q]!$ is defined in equation \eqref{jq}, and $q_{\beta}=q^{\langle \beta,\beta \rangle /2}$. 

\begin{remark}
The proof of the existence of the braiding on the category of finite dimensional weight modules for $U_q^H(\mfg)$ in \cite[Subsection 5.8]{GP1} relies on the action of $\bar{p} \otimes \bar{p}$ on the $R$-matrix $R^h$ of the $h$-adic quantum group $U_h(\mfg)$. However, $R^h$ is an element of the $h$-adic completion $U_h(\mfg) \tilde{\otimes} U_h(\mfg)$, not an element of $U_h(\mfg) \otimes U_h(\mfg)$. This can be easily remedied by extending $\bar{p} \otimes \bar{p}$ to the $h$-adic completion as follows (we refer the reader to \cite[Chapter 16]{K} for any necessary background):

Let $K=\mbbC[[h]]$, for any $K$-module $M$ let $M_j:=M/h^jM$, and for any $K$-linear map $f:M \to M$, define $f_j: M_j \to M_j$ by 
\[ f_j( x+h^jM)=f(x)+h^jM. \]
This map is well-defined since $f$ is $K$-linear, so if $x+h^jM=y+h^jM$, then $x=y+h^jz$ for some $z\in M$, so
\begin{align*}
f_j(x+h^jM)&=f(x)+h^jM\\
&=f(y+h^jz)+h^jM\\
&=f(y)+h^jf(z)+h^jM\\
&=f(y)+h^jM\\
&=f_j(y+h^jM)
\end{align*}
Let $p=\iota \circ \bar{p}:U_h(\mfg) \to U_h(\mfg)$ where $\iota:U^{<} \hookrightarrow U_h(\mfg)$ is the inclusion map. Taking $M=U_h(\mfg) \otimes U_h(\mfg)$ and $f=p \otimes p$, we can construct the associated inverse limit $p \tilde{\otimes} p:U_h(\mfg) \tilde{\otimes} U_h(\mfg) \to U_h(\mfg) \tilde{\otimes} U_h(\mfg)$ where $p \tilde{\otimes} p:=\lim\limits_{\substack{\leftarrow \\ j}}(p \otimes p)_j$ satisfies
\begin{equation}\label{lim} \pi_j \circ (p \tilde{\otimes}p)=(p \tilde{\otimes} p)_j \circ \pi_j\end{equation}
and $\pi_j:U_h(\mfg) \tilde{\otimes}U_h(\mfg) \to (U_h(\mfg) \otimes U_h(\mfg))_j$ is the projection map. Let $R_j^h=\pi_j(R^h)$ denote the image of the $R$-matrix for the $h$-adic quantum group in $(U_h(\mfg) \otimes U_h(\mfg))_j$, then it is easy to see that 
\[ (p\tilde{\otimes} p)_j(R^h_j)=(p \tilde{\otimes} \mathrm{Id})_j(R_j^h)=(\mathrm{Id} \tilde{\otimes} p)_j(R_j^h). \]
for all $j$. We therefore have
\[ \pi_j \circ (p \tilde{\otimes} p)(R^h)=\pi_j \circ (p \tilde{\otimes} \mathrm{Id})(R^h) =\pi_j \circ (\mathrm{Id} \tilde{\otimes} p)(R^h)\]
for all $j$ by Equation \eqref{lim}, so it follows that
\begin{equation}\label{Rcomp} R^{<}:=(p \tilde{\otimes} p)(R^h)=(p \tilde{\otimes} \mathrm{Id})(R^h) = (\mathrm{Id} \tilde{\otimes} p)(R^h) 
\end{equation}
This equation replaces \cite[Equation 42]{GP1}, then the proof that the category of weight modules for $\UHbar{\mfg}$ is braided follows exactly as in \cite[Subsection 5.8]{GP1}, replacing $\otimes$ with $\tilde{\otimes}$ when necessary.
\end{remark}

Let $M\in \mcC$ be simple with maximal vector $m \in M(\lambda)$, and define the family of morphisms $\theta_V:V \to V$ by
\[ \theta_V:= (Id_V \otimes \overleftarrow{\mathrm{ev}}_V) \circ ( c_{V,V} \otimes Id_{V^*} )\circ ( Id_V \otimes \overrightarrow{\mathrm{coev}}_V)\] 
where $c_{V,V}$ is the braiding. An easy computation shows that
\[ \theta_M(m)=q^{\langle \lambda, \lambda+2(1-r) \rho \rangle}m.\]
Hence, on any simple module $M \in \mcC$ with maximal vector $m \in M(\lambda)$, 
\begin{equation} \label{twist}\theta_M=q^{\langle \lambda, \lambda+2(1-r)\rho \rangle} Id_M\\
 \end{equation}

We then observe, as in \cite{GP2}[Subsection 4.4], that $\theta_{(L^{\lambda})^*}=(\theta_{L^{\lambda}})^*$ for all generic simple modules ($L^{\lambda}$ such that $\overline{\lambda} \in \mft \setminus \mcX$), where $f^*$ denotes the right dual. So, by \cite[Theorem 9]{GP2}, $\mcC$ is ribbon.
\begin{corollary}\label{ribbon}
$\mcC$ is a ribbon category.
\end{corollary}
We also observe here that $\mcC$ has trivial M\"{u}ger center (recall Definition \ref{muger}):
\begin{proposition}\label{Muger}
$\mcC$ has no non-trivial transparent objects (i.e. $\mcC$ has trivial M\"{u}ger center).
\end{proposition}
\begin{proof}
We begin by showing that subquotients of transparent objects are transparent. Suppose $Y \in \mcC$ is transparent and $M$ a subobject of $Y$ with embedding $\iota:M \hookrightarrow Y$. Since $Y$ is transparent, we have
\[ c_{Y,X} \circ c_{X,Y}=\mathrm{Id}_{X \otimes Y} \]
for all $X \in \mcC$. Therefore, multiplying both sides of this equation by $\mathrm{Id}_X \otimes \iota$ gives
\[ c_{Y,X} \circ c_{X,Y} \circ (\mathrm{Id}_X \otimes \iota)=\mathrm{Id}_X \otimes \iota. \]
Then by applying naturality of the braiding we find that
\[ (\mathrm{Id}_X \otimes \iota) \circ c_{M,X} \circ c_{X,M} = \mathrm{Id}_X \otimes \iota. \]
$\iota:M \to Y$ is monic so it follows from exactness of the tensor product that $\mathrm{Id} \otimes \iota$ is also monic and therefore
\[ c_{M,X} \circ c_{X,M}=\mathrm{Id}_{X \otimes M} \]
so $M$ is transparent. A similar argument shows that quotients of transparent objects are transparent, and therefore all subquotients of transparent objects are transparent. Given a pair of irreducible modules $L^{\lambda},L^{\mu} \in \mcC$, it is easy to see that the braiding $c=\tau \circ \mcH \tilde{R}$ acts as $\tau \circ \mcH$ on the product of highest weights $v_{\lambda} \otimes v_{\mu} \in L^{\lambda} \otimes L^{\mu}$. We therefore see that
\begin{align*}
c_{L^{\mu},L^{\lambda}} \circ c_{L^{\lambda},L^{\mu}}(v_{\lambda} \otimes v_{\mu})=q^{2\langle \lambda, \mu \rangle} Id.
\end{align*}
Hence, there is no irreducible object transparent to all other irreducible objects since there is no weight $\lambda \in \mfh^*$ such that $\langle \lambda, \mu \rangle \in \frac{\ell}{2}\mathbb{Z}$ for all $\mu \in \mfh^*$. In particular, there are no non-trivial irreducible transparent objects. If some $M \in \mcC$ were transparent, then all of its subquotients and, in particular, the factors appearing in its composition series must also be transparent, so any transparent object has composition factors isomorphic to $\mathds{1}$. Any such module has character $\mathrm{ch}[M]=0$ (recall Equation \eqref{character}), but any module in $\mcC$ with vanishing character is a direct sum of $\mathrm{dim}(M)$ copies of $\mathds{1}$. Indeed, the elements $H_1,...,H_n$ act semisimply on $M$ (as they do on all modules in $\mcC$) so they must act by zero since $\mathrm{ch}[M]=0$, and it follows from Equation \eqref{rootlattice} that $X_{\pm i} m \in M(\pm \alpha_i)= \emptyset$, so $X_{\pm i}m=0$ for all $m\in M$. Recalling that $K_{\gamma}=\prod\limits_{i=1}^nq_i^{k_iH_i}$ as operators on $\mcC$, we see that
\[ X_{\pm i}m=H_im=0 \quad \mathrm{and} \quad K_{\gamma}m=1\]
for all $m \in M$. Hence, $M$ a direct sum of $\mathrm{dim}(M)$ copies of $\mathds{1}$, and $\mcC$ has trivial M\"{u}ger center.
\end{proof}

\subsection{Duality} \label{duality}

Given any $M \in \C$, the antipode $S: \UHbar{\mfg} \to \UHbar{\mfg}$ defines a module structure on the dual $M^*=\mathrm{Hom}_{\mathbb{C}}(M,\mathbb{C})$ by
\[ (x \cdot f) (m)=f(S(x)m) \]

for each $f \in M^*, m \in M$, and $x \in \UHbar{\mfg}$. Let $M^*_{\omega}$ be the module obtained by twisting $M^*$ by the automorphism $\omega$ defined in equation \eqref{eqw}, allowing $x \in \UHbar{\mfg}$ to act on $M^{*}$ as $\omega(x)$ and for convenience set $\check{M}:=M^*_{\omega}$, which is easily seen to lie in $\mcC$. Note that $\check{M}$ is therefore the dual defined with respect to the anti-homomorphism $S \circ \omega: \UHbar{\mfg} \to \UHbar{\mfg}$. It is easy to check that this map is an involution, i.e. $S \circ \omega \circ S \circ \omega = Id$. We therefore have that the canonical map $\phi:M \to \check{\check{M}}$ is an isomorphism, since
\begin{align*}
 X \cdot \phi(v)(f)&=\phi(v)(S (w(X))\cdot f)=\phi(v)(f \circ \Pi(X))=f(X \cdot v)=\phi(X \cdot v)(f),
\end{align*}
where $\Pi:\UHbar{\mfg} \to \mathrm{End}(M)$ is the representation defining the action on $M$. Hence, $\check{\check{M}} \cong M$. The (contravariant) functor $M \mapsto \check{M}$ is exact as the composition of exact functors (taking duals in a tensor category and twisting by automorphisms), and one sees immediately (since $S(H_i)=-H_i$) that $\mathrm{dim}\check{M}(\lambda)=\mathrm{dim}M(\lambda)$, so $\mathrm{ch}[\check{M}]=\mathrm{ch}[M]$. Exactness implies that $M$ is simple iff $\check{M}$ is, and then $\mathrm{ch}[\check{L}^{\lambda}]=\mathrm{ch}[L^{\lambda}]$ implies $\check{L}^{\lambda} \cong L^{\lambda}$. We therefore obtain the following proposition:

\begin{proposition}\label{dfunc}
\begin{itemize}\item[]
\item The contravariant functor $M \mapsto \check{M}$ is exact and $\check{\check{M}} \cong M$.
\item $\mathrm{ch}[\check{M}]=\mathrm{ch}[M]$ for all $M \in \C$.
\item $\check{M}$ is simple iff $M$ is simple.
\item $\check{L}^{\lambda} \cong L^{\lambda}$ for all $\lambda \in \mfh^*$.
\end{itemize}
\end{proposition}

 Recall that a filtration, or series, for a module $M$ is a family of proper submodules ordered by inclusion
\begin{equation*}
0=M_0 \subset M_1 \subset \cdots \subset M_{n-1} \subset M_n=M.
\end{equation*}
A series for a module $M$ is called a composition series if successive quotients are irreducible modules: $M_k/M_{k-1} \cong L^{\lambda_k}$ for some $\lambda_k \in \mfh^*$. Similarly, a series is called a Verma (or standard) series if successive quotients are Verma modules: $M_k/M_{k-1} \cong M^{\lambda_k}$ for some $\lambda_k \in \mfh^*$. 

We have already observed in Proposition \ref{gens} that irreducible modules of typical weight are projective. Given any $V \in \mcC$ and $\lambda \in \mfh^*$ typical, we have a surjection
\begin{equation}\label{enoughproj} \overleftarrow{\mathrm{ev}}_{L^{\lambda}}\otimes Id_V : L^{\lambda} \otimes (L^{\lambda})^* \otimes V \to \mathds{1} \otimes V \cong V \end{equation}
where $L^{\lambda} \otimes (L^{\lambda})^* \otimes V$ is projective since projective modules form an ideal in pivotal categories (see \cite[Lemma 17]{GPV}), so $\mcC$ has enough projectives and since every module in $\mcC$ is finite, every module in $\mcC$ has a projective cover. We denote by $P^{\lambda}$ the projective cover of $L^{\lambda}$, and it follows easily from the defining property of projective modules that $P^{\lambda}$ is also the projective cover of $M^{\lambda}$. Replacing $V$ in Equation \eqref{enoughproj} by an arbitrary Verma module $M^{\mu}$ and noting $M^{\lambda} \cong L^{\lambda}$ for typical $\lambda$ (recall Proposition \ref{irreducible} and \ref{typ}), we obtain a surjection from the projective module $M^{\lambda} \otimes (M^{\lambda})^*\otimes M^{\mu}$ onto $M^{\mu}$. It then follows that $P^{\mu}$ appears in the decomposition of $M^{\lambda} \otimes (M^{\lambda})^* \otimes M^{\mu}$ into a direct sum of projective covers. It can be shown that $M^{\lambda} \otimes (M^{\lambda})^* \otimes M^{\mu}$ has a standard filtration by the argument in \cite[Theorem 3.6]{H}, and $P^{\mu}$ then has a standard filtration by the argument in \cite[Proposition 3.7 (b)]{H} since it is a summand of a module admitting a standard filtration. We denote by $(P^{\lambda}:M^{\mu})$ the multiplicity of $M^{\mu}$ in the standard filtration of $P^{\lambda}$, and $[M^{\mu}:L^{\lambda}]$ the multiplicity of $L^{\lambda}$ in the composition series of $M^{\mu}$. With the existence of the duality functor $M \to \check{M}$ satisfying the properties in Proposition \ref{dfunc}, BGG reciprocity follows as in \cite{H}:

\begin{proposition}\label{func}
BGG reciprocity holds in $\mcC$. That is, we have $(P^{\lambda}:M^{\mu})=[M^{\mu}:L^{\lambda}]$.
\end{proposition}
Let $P \in \mcC$ be projective, then $P$ is isomorphic to a direct sum of projective covers of irreducible modules: $P \cong \bigoplus_{\lambda_k \in \mfh^*} c_{\lambda_k} P^{\lambda_k}$ for some $c_{\lambda_k} \in \mathbb{Z}_+$. Since unrolling the quantum group gives us additive weights, rather than multiplicative weights, the argument of \cite[Corollary 3.10]{H} can be used to show that projectives in $\mcC$ are determined up to isomorphism by their characters, which we include here for convenience. It is clearly enough to show that the characters determine the coefficients $c_{\lambda_k}$ of $P$, since then two projective modules with coinciding characters will both be isomorphic to the same sum of projective covers. We proceed by induction on length of standard filtrations.  If $P$ has length 1, it is a Verma module and the statement is trivial. If $P$ has length $>1$, let 
\[ 0 \subset M_1 \subset \cdots \subset M_n=P \]
with $M_k/M_{k-1} \cong M^{\mu_k}$ denote a standard filtration of $P$, so $\mathrm{ch}[P]=\sum_{i=1}^nd_{\mu_i}\mathrm{ch}[M^{\mu_i}]$. Let $\lambda$ be minimal s.t. $d_{\lambda} \not = 0$. By BGG reciprocity, $(P^{\mu}:M^{\lambda}) \not = 0$ iff $[M^{\lambda}:L^{\mu}] \not = 0$, so $\mu \leq \lambda$ and therefore by minimality of $\lambda$, $P^{\lambda}$ appears in the decomposition of $P$ with multiplicity $d_{\lambda}$, that is, $P=d_{\lambda}P^{\lambda} \oplus \tilde{P}$ for some projective module $\tilde{P}$. By the induction assumption, the coefficients of $\tilde{P}$ are determined by its character, so $P$ is determined by up to isomorphism by its character.
\begin{proposition}
Projective modules in $\mcC$ are isomorphic if their characters coincide.
\end{proposition}
It is apparent from the construction of projective covers in \cite{CGP} for $\overline{U}_q^H(\mathfrak{sl}_2)$ that the socle and top of projective covers coincide. This is actually a general feature of $\mcC$ for any $\overline{U}_q^H(\mfg)$ and is a consquence of the following theorem:

\begin{theorem}\label{selfdualP}
$P^{\lambda}$ is self-dual ($\check{P}^{\lambda} \cong P^{\lambda}$).
\end{theorem}
\begin{proof}
We first recall that any indecomposable module is a quotient of a direct sum of projective covers. Indeed, any cyclic indecomposable module $M$ has a unique maximal submodule and irreducible quotient $L^{\lambda}$. This gives a surjection $\phi:M \twoheadrightarrow L^{\lambda}$ and projectivity of $P^{\lambda}$ guarantees the existence of a map $\varphi:P^{\lambda} \to M$ such that $\phi \circ \varphi=q^{\lambda}$ where $q^{\lambda}:P^{\lambda} \to L^{\lambda}$ is the essential surjection. It follows that there exists a $v \in P^{\lambda}$ such that $\phi \circ \varphi(v)=v_{\lambda} \in L^{\lambda}$ and since $\phi:M \twoheadrightarrow L^{\lambda}$ is the quotient map (by the maximal ideal of $M$), we see that $\varphi(v)$ lies in the top of $M$ and therefore generates $M$ (otherwise it lies in the maximal submodule). Hence, $\varphi:P^{\lambda} \to M$ is surjective. Any indecomposable $M \in \mcC$ is finitely generated by some $\{v_1,...,v_n\}$, so there exists a canonical surjection $\Phi: \bigoplus\limits_{k=1}^n P^{\lambda_k} \to M$ given by mapping each $P^{\lambda_k}$ onto the cyclic submodules $\langle v_k \rangle$ of $M$. We therefore have for each $\check{P}^{\lambda}$ a short exact sequence
\[ 0 \to N^{\lambda} \to \bigoplus\limits_{k=1}^m P^{\lambda_k} \to \check{P}^{\lambda} \to 0\]
for some $\lambda_1,...,\lambda_m \in \mfh^*$ and some submodule $N^{\lambda}$ of $\bigoplus\limits_{k=1}^m P^{\lambda_k}$. Applying the duality functor, we obtain an exact sequence
\[ 0 \to P^{\lambda} \to \bigoplus\limits_{k=1}^m \check{P}^{\lambda_k} \to \check{N}^{\lambda} \to 0.\]
$\mcC$ is pivotal so by \cite[Lemma 17]{GPV}, projective and injective objects coincide in $\mcC$. Therefore, the sequence splits and $P^{\lambda}$ is a summand of $\bigoplus\limits_{k=1}^m \check{P}^{\lambda_k}$. Further, it is easy to see that the functor $X \to \check{X}$ preserves indecomposability since taking duals and twisting by $\omega$ preserve indecomposability in $\mcC$. Since the $\check{P}^{\lambda_k}$ are indecomposable, we have $P^{\lambda} \cong \check{P}^{\lambda_k}$ for some $\lambda_k \in \mfh^*$ and $\mathrm{ch}[\check{P}^{\lambda}]=\mathrm{ch}[P^{\lambda}$], so we must have $\lambda_k=\lambda$. That is, $\check{P}^{\lambda} \cong P^{\lambda}$.
\end{proof}
Theorem \ref{selfdualP} has the following immediate corollary
\begin{corollary}\label{projcorollary}
\begin{itemize}
\item $\mathrm{Socle}(P^{\lambda})=L^{\lambda}$.
\item $P^{\lambda}$ is the injective hull of $L^{\lambda}$.
\item $\mcC$ is Unimodular.
\end{itemize}
\end{corollary}
Unimodularity follows from $P^0$ being the injective hull of $\mathds{1}$ (see \cite{ENO,EGNO}), so by \cite[Corollary 3.2.1]{GKP}, we see that there exists a right trace on the ideal of projective modules in $\mcC$ (for details on categorical traces see \cite[Subsection 1.3]{GP2}.). It then follows exactly as in \cite[Theorem 22]{GP2} that this right trace is in fact a two-sided trace:

\begin{corollary}\label{tracecorollary}
 $\mcC$ admits a unique non-zero two-sided trace on the ideal $Proj$ of projective modules.
\end{corollary}

\end{document}

%% file: preamble.tex
\usepackage{amsmath,amsthm,amssymb,amsfonts,amscd, mathrsfs, mathtools, dsfont}
\usepackage[hyperindex,breaklinks]{hyperref}
\usepackage{cleveref}
\usepackage{tikz}
\usepackage{tikz-cd}
\usepackage{cancel}
\usepackage{fullpage}
\usepackage{enumerate}
\usepackage[shortlabels]{enumitem}
\usepackage{makecell}
\usetikzlibrary{matrix,arrows,decorations.pathmorphing}

\theoremstyle{definition}
\newtheorem{theorem}{Theorem}[section]
\newtheorem{corollary}[theorem]{Corollary}

\newtheorem{proposition}[theorem]{Proposition}

\newtheorem{definition}[theorem]{Definition}

\newtheorem{remark}[theorem]{Remark}
\newtheorem*{acknowledgements}{Acknowledgements}

\numberwithin{equation}{section}
\numberwithin{theorem}{section}

\definecolor{darkgreen}{rgb}{0,0.5,0}
\definecolor{darkblue}{rgb}{0,0.1,0.5}

\hypersetup{
    colorlinks=true, % make the links colored
    linkcolor=darkblue, % color TOC links in blue
    urlcolor=blue, % color URLs in red
    linktoc=all, % 'all' will create links for everything in the TOC
    citecolor=darkgreen
    }

\newcommand{\Uq}[1]{U_q(#1)}
\newcommand{\UH}[1]{U^{H}_q(#1)}
\newcommand{\UHbar}[1]{\overline{U}_q^{H}(#1)}

\newcommand{\mfg}{\mathfrak{g}}
\newcommand{\mfh}{\mathfrak{h}}

\newcommand{\mft}{\mathfrak{t}}

\newcommand{\mcB}{\mathcal{B}}
\newcommand{\mcC}{\mathcal{C}}

\newcommand{\mcG}{\mathcal{G}}
\newcommand{\mcH}{\mathcal{H}}

\newcommand{\mcO}{\mathcal{O}}

\newcommand{\mcX}{\mathcal{X}}

\newcommand{\mbbC}{\mathbb{C}}

\newcommand{\mbbZ}{\mathbb{Z}}

%% file: Unrolled-Final.bbl
\begin{thebibliography}{CGP00}

\bibitem[A]{A} D. Adamovic, Classification of irreducible modules of certain subalgebras of free boson vertex algebra, \doi{10.1016/j.jalgebra.2003.07.011}{Journal of Algebra, Volume 270, Issue 1, 1 December 2003, Pages 115-132.}

\bibitem[ACKR]{ACKR} J. Auger, T. Creutzig, S. Kanade, M. Rupert, Braided Tensor Categories related to Bp Vertex Algebras, \href{https://doi.org/10.1007/s00220-020-03747-8}{Communications in Mathematical Physics, Volume 378, pages 219–260 (2020)}.

\bibitem[AM1]{AM1} D. Adamovic and A. Milas, On the triplet vertex algebra $W(p)$, \href{https://doi.org/10.1016/j.aim.2007.11.012}{Advances in Mathematics, Volume 217, Issue 6, 1 April 2008, Pages 2664-2699.}

\bibitem[AM2]{AM3} D. Adamovic and A. Milas , Logarithmic intertwining operators and $W(2,2p-1)$-algebras, \href{https://doi.org/10.1063/1.2747725}{Journal of Mathematical Physics , 48 073503, 2007.}

\bibitem[AM3]{AM4} D. Adamovic and A. Milas, Some applications and constructions of intertwining operators in LCFT, \doi{10.1090/conm/695/13992}{Contemporary Mathematics, Volume 695, 2017.}


\bibitem[AS]{AS} N. Andruskiewitsch and C. Schweigert, On Unrolled Hopf Algebras, \href{https://doi.org/10.1142/S0218216518500530}{Journal of Knot Theory and Its Ramifications Vol. 27, No. 10, 1850053, 2018.}

\bibitem[B]{B} J. Bichon, Cosovereign Hopf Algebras, \href{https://doi.org/10.1016/S0022-4049(00)00024-4}{Journal of Pure and Applied Algebra, Volume 157, Issues 2–3, 23 March 2001, Pages 121-133.}

\bibitem[BCGP]{BCGP} C. Blanchet, F. Costantino, N. Geer, and B. Patureau-Mirand, Non Semi-Simple TQFTs from Unrolled Quantum sl(2), \href{https://www.intlpress.com/site/pub/pages/books/items/00000468/index.php}{Proceedings of the Gokova Geometry-Topology Conference 2015, International Press of Boston, Inc, 2016}, \href{https://arxiv.org/pdf/1605.07941.pdf}{[arXiv:1605.07941].}

\bibitem[BM]{BM} K. Bringmann and A. Milas, W-algebras, higher rank false theta functions, and quantum dimensions, \href{ https://doi.org/10.1007/s00029-016-0289-z}{Selecta Math. New Ser. (2017) 23: 1249.}

\bibitem[C]{C} T. Creutzig - Logarithmic W-algebras and Argyres-Douglas theories at higher rank, \href{https://doi.org/10.1007/JHEP11(2018)188}{J. High Energ. Phys. (2018) 2018: 188. }

\bibitem[CGP]{CGP} F. Costantino, N. Geer, B. Patureau-Mirand, Some Remarks on the Unrolled Quantum Group of $\mathfrak{sl}(2)$, \href{https://doi.org/10.1016/j.jpaa.2014.10.012}{Journal of Pure and Applied Algebra, Volume 219, Issue 8, August 2015, Pages 3238-3262.}

\bibitem[CGP2]{CGP2} F. Costantino, N. Geer, B. Patureau-Mirand, Quantum invariants of 3-manifolds via link surgery presentations and non-semi-simple categories, \href{https://doi.org/10.1112/jtopol/jtu006}{Journal of Topology, Volume 7, Issue 4, December 2014, Pages 1005–1053,.}

\bibitem[CGR]{CGR} T. Creutzig, A. Gainutdinov, and I. Runkel, A quasi-Hopf algebra for the triplet vertex operator algebra, \href{https://doi.org/10.1142/S021919971950024X}{Communications in Contemporary Mathematics 1950024, 2019.}

\bibitem[CM1]{CM1} T. Creutzig and A. Milas, False theta functions and the Verlinde formula, \href{https://doi.org/10.1016/j.aim.2014.05.018}{Advances in Mathematics Volume 262, 10 September 2014, Pages 520-545.}

\bibitem[CM2]{CM2} T. Creutzig and A. Milas, Higher rank partial and false theta functions and representation theory, \href{https://doi.org/10.1016/j.aim.2017.04.027}{Advances in Mathematics, Volume 314, 9 July 2017, Pages 203-227.}

\bibitem[CMR]{CMR} T. Creutzig, A. Milas, and M. Rupert, Logarithmic Link invariants of $\overline{U}_q^H(\mathfrak{sl}_2)$ and asymptotic dimensions of singlet vertex algebras, \href{https://doi.org/10.1016/j.jpaa.2017.12.004}{Journal of Pure and Applied Algebra, Volume 222, Issue 10, October 2018, Pages 3224-3247.}

\bibitem[CMW]{CMW} T. Creutzig, A. Milas, and S. Wood, On Regularised Quantum Dimensions of the Singlet Vertex Operator Algebra and False Theta Functions, \doi{10.1093/imrn/rnw037}{International Mathematics Research Notices, Volume 2017, Issue 5, 1 March 2017, Pages 1390–1432.}

\bibitem[CP]{CP} V. Chari, A. Pressley, A Guide to Quantum Groups, pp. 667. ISBN 0521558840. Cambridge, UK: Cambridge University Press, October 1995.

\bibitem[D]{D} M. De Renzi, Non-Semisimple Extended Topological Quantum Field Theories \href{https://arxiv.org/pdf/1703.07573.pdf}{[arXiv:1703.07537]}.


\bibitem[DCK]{DCK} C. De Concini and V. Kac, Representations of Quantum Groups at Roots of 1, In Operator algebras, unitary representations, enveloping algebras, and invariant theory. (Paris, 1989), 471-506, Progr. Math., 92, Birkhauser
Boston, 1990.

\bibitem[DGP]{DGP} M. De Renzi, N. Geer, B. Patureau-Mirand, Non-Semisimple Quantum Invariants and TQFTS
from Small and Unrolled Quantum Groups, \href{https://doi.org/10.2140/agt.2020.20.3377}{Algebraic \& Geometric Topology 20 (2020) 3377–3422.}

\bibitem[EGNO]{EGNO} P. Etingof, S. Gelaki, D. Nikshych, V. Ostrik, Tensor Categories, American Mathematical Society, Mathematical Surverys and Monographs, Vol. 205, 2015.

\bibitem[ENO]{ENO} P. Etingof, D. Nikshych, V. Ostrik, An analogue of Radford's $S^4$ formula for finite tensor categories, \href{https://doi.org/10.1155/S1073792804141445}{International Mathematics Research Notices, Volume 2004, Issue 54, 2004, Pages 2915–2933.}

\bibitem[FT]{FT} B. Feigin, I. Tipunin, Logarithmic CFTs connected with simple Lie algebras, \href{https://arxiv.org/abs/1002.5047}{[arXiv:1002.5047]}.

\bibitem[FL]{FL}
		I.~Flandoli and S.~Lentner, Logarithmic conformal field theories of type $B_n,\ell=4$ and symplectic fermions, \href{https://doi.org/10.1063/1.5010904}{J. Math. Phys. \textbf{59} (2018) no.7, 071701}.

\bibitem[GKP]{GKP} N. Geer, J. Kujawa, B. Patureau-Mirand, Ambidextrous Objects and Trace Functions for NonSemisimple Categories, \href{https://doi.org/10.1090/S0002-9939-2013-11563-7}{Proceedings of the AMS, Volume 141, No.9, September 2013, Pages 2963–2978.}

\bibitem[GLO]{GLO} A. Gainutdinov, S. Lentner, T. Ohrmann, Modularization of small quantum groups, \href{https://arxiv.org/pdf/1809.02116.pdf}{[arXiv:1809.02116v2].}

\bibitem[GP1]{GP1} N. Geer, B. Patureau-Mirand, Topological invariants from non-restricted quantum groups, \doi{10.2140/agt.2013.13.3305}{Algebraic \& Geometric Topology 13 (2013) 3305–3363.}

\bibitem[GP2]{GP2} N. Geer, B. Patureau-Mirand, The trace on projective representations of quantum groups, \doi{10.1007/s11005-017-0993-4}{Lett Math Phys (2018) 108: 117.}

\bibitem[GPT1]{GPT1} N. Geer B. Patureau-Mirand, V. Turaev, Modified quantum dimensions and re-normalized link invariants, \href{https://doi.org/10.1112/S0010437X08003795}{Compositio Mathematica, Volume 145, Issue 1, January 2009 , pp. 196-212}

\bibitem[GPT2]{GPT2} N. Geer, B. Patureau-Mirand, V. Turaev, Modified 6j-Symbols and 3-manifold invariants. \href{https://www.sciencedirect.com/science/article/pii/S0001870811002040}{Advances in Mathematics Volume 228, Issue 2, (2011), 1163-1202.}

\bibitem[GPV]{GPV} N. Geer, B. Patureau-Mirand, A. Virelizier, Traces on ideals in pivotal categories, \doi{10.4171/QT/36}{Quantum Topol. 4 (2013), 91-124.}


\bibitem[Hu]{H} J. Humphreys, Representations of Semisimple Lie Algebras in the BGG Category $\mcO$,  Graduate Studies in Mathematics, Volume 94, American Mathematical Society, 2008.

\bibitem[J]{J} M. Jimbo, A $q$-difference analogue of $U(\mfg)$ and the Yang-Baxter equation, \doi{10.1007/BF00704588}{Lett Math Phys (1985) 10: 63.}

\bibitem[K]{K} C. Kassel, Quantum Groups, Graduate Texts in Mathematics 155, Berlin: Springer Verlag, 1995

\bibitem[KS]{KS} H. Kondo and Y. Saito, Indecomposable decomposition of tensor products of
modules over the restricted quantum universal enveloping algebra associated to $\mathfrak{sl}_2$, \href{https://doi.org/10.1016/j.jalgebra.2011.01.010}{Journal of Algebra, Volume 330, Issue 1, 15 March 2011, Pages 103-129.}

\bibitem[KlS]{KlS} A. Klimyk, K. Schmüdgen, Quantum Groups and Their Representations, Springer-Verlag Berlin Heidelberg, 1997.

\bibitem[L]{L} G. Lusztig, Quantum groups at roots of 1, \href{https://link.springer.com/article/10.1007/BF00147341}{Geometriae Dedicata, Volume 35, pages 89–113(1990)}.

\bibitem[Len1]{Len1} S. Lentner, A Frobenius homomorphism for Lusztig’s quantum groups for arbitrary roots of unity, \href{https://doi.org/10.1142/S0219199715500406}{Communications in Contemporary Mathematics, Vol. 18, No. 03, 1550040 (2016).}

\bibitem[Len2]{Len2} S. Lentner, The unrolled quantum group inside Lusztig’s quantum group of divided powers, \href{https://doi.org/10.1007/s11005-019-01185-9}{Letters in Mathematical Physics volume 109, pages1665–1682 (2019).}

\bibitem[Len3]{Len3} S.~D.~Lentner, Quantum groups and Nichols algebras acting on conformal field theories, \href{https://doi.org/10.1016/j.aim.2020.107517}{Adv. Math \textbf{378} (2021) 			107517}.

\bibitem[Mi]{Mi} A. Milas, Characters of Modules of Irrational Vertex Algebras, \href{https://doi.org/10.1007/978-3-662-43831-2_1}{Kohnen W., Weissauer R. (eds) Conformal Field Theory, Automorphic Forms and Related Topics. Contributions in Mathematical and Computational Sciences, vol 8. Springer, Berlin, Heidelberg (2014).}

\bibitem[N]{N} C. Negron - Log-modular quantum groups at even roots of unity and the quantum Frobenius I, \href{https://doi.org/10.1007/s00220-021-04012-2}{Communications in Mathematical Physics volume 382, pages773–814 (2021).}

\bibitem[O]{O} T. Ohtsuki - Quantum invariants. A study of knots, 3-manifolds, and their sets. Series on Knots and
Everything, 29. World Scientific Publishing Co., Inc., River Edge, NJ, 2002.

\bibitem[S]{S} K. Shimizu, Non-degeneracy conditions for braided finite tensor categories, \doi{https://doi.org/10.1016/j.aim.2019.106778}{Advances in Mathematics Volume 355, 15 October 2019, 106778}.


\end{thebibliography}
